\documentclass[12pt]{article}
\usepackage[paperwidth=8.5in, paperheight=11in, top=1in, bottom=1in, left=.5in, right=.5in]{geometry}

\usepackage{amsmath,amsfonts,amssymb,epsfig,graphicx,latexsym}
\usepackage{xspace,multicol,paracol,caption,bigstrut}
\usepackage{theorem}
\usepackage{color}

\newtheorem{theorem}{Theorem}[section]
\newtheorem{corollary}[theorem]{Corollary}
\newtheorem{lemma}[theorem]{Lemma}
\newtheorem{question}{Question}

\numberwithin{equation}{section}

\newcommand{\ds}{\displaystyle}

\newenvironment{proof}[1][Proof]{\begin{trivlist}
\item[\hskip \labelsep {\bfseries #1}]}{\end{trivlist}}

\newcommand{\qed}{\hfill \ensuremath{\Box}}

\begin{document}

\thispagestyle{empty}
\title{\textbf{Assorted Musings on Dimension-critical Graphs}}

\author{\textbf{Matt Noble} \\
Department of Mathematics\\
Middle Georgia State University\\
matthew.noble@mga.edu}
\date{}
\maketitle

\begin{abstract} 

\vspace{3pt}

For a finite simple graph $G$, say $G$ is of dimension $n$, and write $\dim(G) = n$, if $n$ is the smallest integer such that $G$ can be represented as a unit-distance graph in $\mathbb{R}^n$.  Define $G$ to be \emph{dimension-critical} if every proper subgraph of $G$ has dimension less than $G$.  In this article, we determine exactly which complete multipartite graphs are dimension-critical.  It is then shown that for each $n \geq 2$, there is an arbitrarily large dimension-critical graph $G$ with $\dim(G) = n$.  We then pose and expound upon a number of questions related to this subject matter, questions that hopefully will prompt future research.\\[5pt]

\noindent \textbf{Keywords and phrases:} graph dimension, unit-distance graph embeddings, edge-criticality, complete multipartite graphs
\end{abstract}

\section{Introduction}

Define a finite simple graph $G$ to be \emph{representable} (or alternately \emph{embeddable}) in $\mathbb{R}^n$ if $G$ can be drawn with its vertices being points of $\mathbb{R}^n$ where any two adjacent vertices are necessarily placed at a unit-distance apart.  Say $G$ is of dimension $n$, and denote $\dim(G) = n$, if $G$ is representable in $\mathbb{R}^n$ but not in $\mathbb{R}^{n-1}$.  For a non-empty graph $G$, define $G$ to be \emph{dimension-critical} if for every proper subgraph $H$ of $G$, $\dim(H) < \dim(G)$.  

This notion of graph dimension was initially put forth in a 1965 note by Erd\H{o}s, Harary, and Tutte \cite{eht}.  There the authors establish the dimension of a few common families of graphs and, as typical of a paper authored or co-authored by Erd\H{o}s, conclude by stirring the pot with a number of questions for future investigation.  Indeed, one of these questions serves as an impetus for our present work.  Erd\H{o}s, Harary, and Tutte ask the reader to ``\ldots characterize the critical $n$-dimensional graphs, at least for $n = 3$ (this is trivial for $n = 2$)."  Indeed, it takes only a moment's thought to conclude that if $G$ is a dimension-critical graph with $\dim(G) = 2$, then $G$ is either a cycle or the star $K_{1,3}$.  For higher dimensions, the situation is murkier, and for an arbitrary graph $G$, an efficiently-computed condition that is both necessary and sufficient for $G$ to be dimension-critical seems unlikely to exist.  We can, however, claim success in characterizing the criticality of certain families of graphs.  

In Section 2, we give a full description of which complete multipartite graphs are dimension-critical.  To more succinctly phrase our result, we will implement notation similar to that seen in \cite{maehara}.  For non-negative integers $\alpha$, $\beta$, and $\gamma$, define $G(\alpha, \beta, \gamma)$ to be the complete multipartite graph with $\alpha + \beta + \gamma$ parts, $\alpha$ of which are of size 1, $\beta$ of which are of size 2, and $\gamma$ of which are of size 3.  We first observe that any complete multipartite graph having a part of size 4 or larger is in fact not dimension-critical, and then determine exactly which assignments of $\alpha$, $\beta$, and $\gamma$ result in $G(\alpha, \beta, \gamma)$ being dimension-critical.

In Section 3, for any $n \geq 2$ and positive integer $c$, we show through an explicit construction the existence of a dimension-critical graph $G$ with $\dim(G) = n$ and $|E(G)| > c$.  This generalizes a result of Boza and Revuelta \cite{boza} where they show it is possible for $n = 3$.  

In Section 4, we conclude with a number of observations and questions that will hopefully re-stir the pot and prompt future research.  In particular, we make a beginning foray into the problem of determining for which $n,k$ there exists an arbitrarily large dimension-critical graph $G$ having $\dim(G) = n$ and chromatic number $\chi(G) = k$.

\section{Dimension-critical Complete Multipartite Graphs} 

In \cite{maehara}, Maehara determines the Euclidean dimension of all complete multipartite graphs.  We ourselves will not be concerned with this particular graph parameter, however for those interested readers, we remark that the Euclidean dimension of a graph $G$ is defined similarly to the dimension of $G$ with the added stipulation that in any representation in $\mathbb{R}^n$, non-adjacent vertices of $G$ are forbidden to be placed a unit-distance apart.  Regardless, the following theorem is an easily established corollary of the work done in \cite{maehara}.

\begin{theorem} \label{hiroshi} Let $G$ be a complete $(\alpha + \beta + \gamma)$-partite graph having exactly $\alpha$ parts of size 1, exactly $\beta$ parts of size 2, and exactly $\gamma$ parts of size greater than or equal to 3.  If $\beta + \gamma \leq 1$, then $\dim(G) = \alpha + \beta + 2\gamma - 1$.  If $\beta + \gamma \geq 2$, $\dim(G) = \alpha + \beta + 2\gamma$.
\end{theorem}

Theorem \ref{hiroshi} will figure prominently in this section, and indeed it has an immediate and relevant bearing.  Letting $G$ be a complete multipartite graph containing a part of size 4 or larger, and letting $G'$ be the graph formed by deleting from $G$ a vertex of that part, we have that $\dim(G) = \dim(G')$.  This gives us the corollary below.

\begin{corollary} \label{part4} Let $G$ be a complete multipartite graph having at least one part of size 4 or larger.  Then $G$ is not dimension-critical.
\end{corollary}

Now let $G$ be equal to some $G(\alpha, \beta, \gamma)$, and let $e \in E(G)$.  In deciding whether or not $G$ is dimension-critical, we will often consider $G - e$ as a subgraph of some other complete multipartite graph.  As an example, consider $G = G(1,0,2)$.  Label the partite sets of $G$ as $\{a\}, \{b_1, c_1, d_1\}, \{b_2, c_2, d_2\}$, and let $e = b_1b_2$.  Then $G - e$ is a subgraph of $G(1,3,0)$ whose partite sets are $\{a\}, \{b_1, b_2\}, \{c_1, d_1\}, \{c_2, d_2\}$.  

We list those dimension-critical complete multipartite graphs in the theorem below.

\begin{theorem} \label{if} Each of the following complete multipartite graphs are dimension-critical.
\begin{enumerate}
\item[(i)] $K_{\alpha}$ for $\alpha \geq 3$
\item[(ii)] $C_4$
\item[(iii)] $K_{1,3}$
\item[(iv)] $K_{2,3}$
\item[(v)] $G(\alpha,0,\gamma)$ for $\alpha \geq 0$ and $\gamma \geq 2$
\end{enumerate}
\end{theorem}

\begin{proof} In \cite{eht}, it is observed that $\dim(K_{\alpha}) = \alpha - 1$ and that $\dim(K_\alpha - e) = \alpha - 2$, so we have that $K_{\alpha}$ is dimension-critical.  It is obvious that the cycle $C_4$ and star $K_{1,3}$ are dimension-critical.  It is also fairly easy to see that $\dim(K_{2,3}) = 3$ and $\dim(K_{2,3} - e) = 2$, although it is noted as well in \cite{joe6} that $K_{2,3}$ is a dimension 3 graph with minimum edge-set, which in turn implies that $K_{2,3}$ is dimension-critical.  

Now let $G = G(\alpha,0,\gamma)$ for $\alpha \geq 0$ and $\gamma \geq 2$, and note that Theorem \ref{hiroshi} gives $\dim(G) = \alpha + 2\gamma$.  Label the partite sets of $G$ as $\{a_1\}, \ldots, \{a_\alpha\}$, $\{b_1, c_1, d_1\}, \ldots, \{b_\gamma, c_\gamma, d_\gamma\}$.  Let $e_1 = a_1a_2$, $e_2 = b_1b_2$, and $e_3 = a_1b_1$.  For any $e \in E(G)$, there is an automorphism of $G$ mapping $e$ to one of $e_1$, $e_2$, or $e_3$, so to show that $G$ is dimension-critical, we just need to show that for $i \in \{1,2,3\}$, $\dim(G) > \dim(G - e_i)$.  First note that $G - e_1$ is a subgraph of $G(\alpha - 2, 1, \gamma)$ which by Theorem \ref{hiroshi} is of dimension $\alpha + 2\gamma - 1$.  Secondly, note that $G - e_2$ is a subgraph of $G(\alpha, 3, \gamma - 2)$ which again by Theorem \ref{hiroshi} is of dimension $\alpha + 2\gamma - 1$.  Finally, we have that $G - e_3$ is a subgraph of $G(\alpha - 1, 2, \gamma - 1)$ which is of dimension $\alpha + 2\gamma - 1$ as well.  Since for arbitrary graphs $H$ and $K$, $H$ being a subgraph of $K$ implies that $\dim(H) \leq \dim(K)$, we have now shown that for any $e \in E(G)$, $\dim(G - e) \leq \alpha + 2\gamma - 1 < \dim(G)$.  This completes the proof that $G$ is dimension-critical.\qed
\end{proof}

\begin{theorem} \label{betagamma3} Let $G = G(\alpha, \beta, \gamma)$ where $\alpha \geq 0$, $\beta \geq 1$, and $\beta + \gamma \geq 3$.  Then $G$ is not dimension-critical.
\end{theorem}

\begin{proof} Let $v \in V(G)$ where $v$ is contained in a part of size 2.  Then $G \setminus \{v\}$ = $G(\alpha + 1, \beta - 1, \gamma)$ and by Theorem \ref{hiroshi}, $\dim(G \setminus \{v\}) = \alpha + \beta + 2\gamma$.  Since $\dim(G) = \alpha + \beta + 2\gamma$ as well, we have that $G$ is not dimension-critical.\qed
\end{proof}

In light of Theorems \ref{if} and \ref{betagamma3}, the only remaining complete multipartite graphs that we must investigate are $K_2$, $G(\alpha, 1 , 0)$ for $\alpha \geq 1$, $G(\alpha, 1, 1)$ for $\alpha \geq 1$, $G(\alpha, 2 , 0)$ for $\alpha \geq 1$, and $G(\alpha, 0 , 1)$ for $\alpha \geq 2$.  We show that each of these graphs are not dimension-critical in the theorem below.

\begin{theorem} \label{notcritical} Each of the following complete multipartite graphs are not dimension-critical. 
\begin{enumerate}
\item[(i)] $K_2$
\item[(ii)] $G(\alpha, 1 , 0)$ for $\alpha \geq 1$
\item[(iii)] $G(\alpha, 1, 1)$ for $\alpha \geq 1$
\item[(iv)] $G(\alpha, 2 , 0)$ for $\alpha \geq 1$
\item[(v)] $G(\alpha, 0 , 1)$ for $\alpha \geq 2$
\end{enumerate}
\end{theorem}
\begin{proof} We consider each of these cases individually, and apply Theorem \ref{hiroshi} throughout.
\begin{enumerate}
\item[(i)] Quite obviously $\dim(K_2) = 1$, however deletion of the only edge of $K_2$ results in a graph just consisting of two isolated vertices which cannot be embedded in $\mathbb{R}^0$ (which by convention consists of a single point).  So $K_2$ is not dimension-critical.

\item[(ii)] Let $G = G(\alpha, 1 , 0)$ for $\alpha \geq 1$.  Then $\dim(G) = \alpha$.  Letting $v \in V(G)$ where $v$ is contained in the part of size 2, $G \setminus \{v\}$ is equal to $K_{\alpha + 1}$.  Since $\dim(K_{\alpha + 1}) = \alpha$, we have that $G$ is not dimension-critical.

\item[(iii)] Let $G = G(\alpha, 1, 1)$ for $\alpha \geq 1$, and note that $\dim(G) = \alpha + 3$.  Label the partite sets of $G$ as $\{a_1\}, \ldots, \{a_\alpha\}, \{b_1,c_1\}, \{b_2,c_2,d_2\}$.  Form a new graph $G'$ by deleting from $G$ the edges $a_1b_1$ and $a_1c_1$.  Observe that $G' = G(\alpha - 1, 0, 2)$ and $\dim(G') = \alpha + 3$.  Again, we have that $G$ is not dimension-critical.

\item[(iv)] Let $G = G(\alpha, 2 , 0)$ for $\alpha \geq 1$, and note that $\dim(G) = \alpha + 2$.  Just as in the last case, let $G' = G \setminus \{a_1b_1, a_1c_1\}$ and note that $G' = G(\alpha - 1, 1, 1)$.  We have that $\dim(G') = \alpha + 2$ which implies that $G$ is not dimension-critical.

\item[(v)] Finally, let $G = G(\alpha, 0 , 1)$ for $\alpha \geq 2$, which gives $\dim(G) = \alpha + 1$.  Let $\{a_1\}$ and $\{a_2\}$ both be partite sets of size 1, and let $G'$ be formed by deleting edge $a_1a_2$ from $G$.  Then $G' = G(\alpha - 2, 1, 1)$ and $\dim(G') = \alpha + 1$.  We conclude that $G$ is not dimension-critical.\qed
\end{enumerate}
\end{proof}

Theorem \ref{notcritical} concludes that the graphs shown to be dimension-critical by Theorem \ref{if} are in fact the only dimension-critical complete multipartite graphs.

\section{Arbitrarily Large Dimension-critical Graphs}

In this section, we show that for any $n \geq 2$, there exists an arbitrarily large dimension-critical graph $G$ with $\dim(G) = n$.  This is immediate for $n = 2$ as the cycle $C_k$ is of dimension 2 for any $k \geq 3$, and deletion of any edge of $C_m$ results in a path which has a unit-distance representation on the real number line $\mathbb{R}$.  In \cite{boza}, Boza and Revuelta construct an arbitrarily large dimension-critical graph of dimension 3.  However, the authors of \cite{boza} do not comment on the existence of such graphs in higher dimensions, and it does not appear that their construction has a clear generalization.

We will obtain our result by considering the graph $G = K_n + C_m$.  That is, $G$ is formed by starting with the cycle $C_m$ for some $m \geq 3$, then placing $n$ vertices adjacent to each other and to each of the vertices of the copy of $C_m$.  Along the way, we will employ a number of lemmas and theorems of a geometric sort.  Lemma \ref{erdos1} is observed in the previously mentioned \cite{eht}. 

\begin{lemma} \label{erdos1} For any $n \geq 1$, $\dim(K_n) = n - 1$.
\end{lemma}

Regarding Lemma \ref{erdos1}, it is well-known that if $n$ points are equidistant in $\mathbb{R}^{n-1}$, then these points must constitute the vertices of a regular $(n-1)$-dimensional simplex.  This representation of $K_n$ in $\mathbb{R}^{n-1}$ is unique up to Euclidean movements -- that is, unique up to rotations, reflections, and translations.  This fact will play a key role in our current work.  Lemma \ref{simplex} is also well-known and can be found, for example, in \cite{brass}.  

\begin{lemma} \label{simplex} Let $S$ be a regular $n$-dimensional simplex embedded on a unit sphere.  Then for any vertices $P_1$, $P_2$ of $S$, $|P_1 - P_2| = \sqrt{2 + \frac{2}{n}}$.
\end{lemma}

Since, in any representation of a graph in $\mathbb{R}^n$, we require edges to be of unit length, a quick calculation allows Lemma \ref{simplex} to be restated as the following corollary.

\begin{corollary} \label{sphere} Let $K_n$ have a unit-distance representation in $\mathbb{R}^{n-1}$ on a sphere of radius $r$.  Then $r = \sqrt{\frac{n-1}{2n}}$.
\end{corollary}

Theorem \ref{angle} is found in \cite{varona} and will be implemented in the proof of Lemma \ref{nocycle} below.

\begin{theorem} \label{angle} Let $r \in \mathbb{Q}$ with $0 \leq r \leq 1$.  The number $\frac{1}{\pi}\arcsin(\sqrt{r})$ is rational if and only if $r$ is equal to  $0$, $\frac14$, $\frac12$, $\frac34$, or 1.
\end{theorem}

\begin{lemma} \label{nocycle} Let $S$ be a circle of radius $r = \sqrt{\frac{n + 1}{2n}}$ for some integer $n \geq 2$.  Then no cycle of edge-length 1 is embeddable on $S$.
\end{lemma}

\begin{proof} Consider a cycle $C_m$, and assume to the contrary that $C_m$ is embeddable on $S$.  We then must have that, for some integer $z \in \mathbb{Z}^+$, the angle $\theta$ given in Figure \ref{angle} satisfies $m\theta = z(2\pi)$.

\begin{figure}[h]%
\begin{center}
\includegraphics[scale=.65]{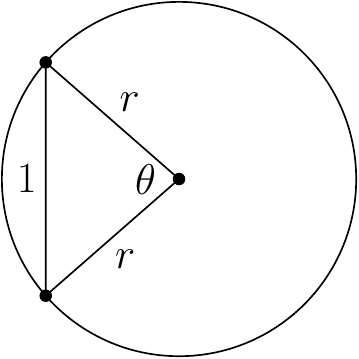}
\caption{}%
\label{angle}%
\end{center}
\end{figure}

\noindent Solving for $\theta$, we have $\sin(\frac{\theta}{2}) = \frac{1}{2r}$.  Combining this with the equality given above, we have that\\ $\frac{1}{\pi}\arcsin \left(\sqrt{\frac{n}{2(n+1)}}\right) = \frac{z}{m}$, or in other words, $\frac{1}{\pi}\arcsin \left(\sqrt{\frac{n}{2(n+1)}}\right)$ is rational.  By Theorem \ref{angle}, we see $\frac{n}{2(n+1)} \in \{\frac14, \frac12, \frac34, 1\}$.  However, letting $f(x) = \frac{x}{2(x+1)}$, we have $f'(x) = \frac{1}{2(x+1)^2}$ which implies that $f(x)$ is strictly increasing.  Since $f(2) = \frac13$ and $\ds \lim_{n\to \infty} f(n) =$ $\frac12$, we have a contradiction.\qed
\end{proof}

We now determine the dimension of $G = K_n + C_m$.

\begin{theorem} \label{knpluscm} Let $G = K_n + C_m$ for $m \geq 3$, $n \geq 2$.  Then $\dim(G) = n + 2$.
\end{theorem}

\begin{proof} Label the vertices of $K_n$ as $a_1, \ldots, a_n$ and those of $C_m$ as $w_1, \ldots, w_m$.  Our first goal is to find an embedding of $G$ in $\mathbb{R}^{n+2}$.  We can represent $a_1, \ldots, a_n$ as the vertices of a regular $(n-1)$-dimensional simplex of edge-length 1 centered at the origin.  Let $r_1$ be the radius of this simplex, and by Corollary \ref{sphere}, we have $r_1 = \sqrt{\frac{n-1}{2n}}$.  Each of the vertices $w_1, \ldots, w_m$ will then be represented as $\mathbb{R}^{n+2}$ points of the form $(0, \ldots, 0, x_i, y_i, z_i)$ where $x_i^2 + y_i^2 + z_i^2 = 1 - r_1^2$ for $i \in \{1, \ldots, m\}$.  Let $r_2 = \sqrt{1 - r_1^2}$, and note that $r_2 = \sqrt{\frac12 + \frac{1}{2n}}$.  To complete our embedding of $G$ in $\mathbb{R}^{n+2}$, it now suffices to show that the cycle $C_m$ is representable in $\mathbb{R}^3$ with each of its vertices lying on a sphere of radius $r_2$.

Designate by $S$ a sphere of radius $r_2$.  First, we claim that for any points $P_1$, $P_2$ lying on $S$ with $|P_1 - P_2| = 1$, there exists a point $P_3$ on $S$ at distance 1 from each of $P_1$, $P_2$.  To see this, we will show that there exists a point on $S$ at distance less than 1 from each of $P_1$, $P_2$ and also a point on $S$ at distance greater than 1 from each of $P_1$, $P_2$, whereby continuity guarantees the existence of the desired $P_3$.  Consider the great circle of $S$ containing both $P_1$ and $P_2$, and then label distances as in Figure \ref{circle} below.

\begin{figure}[h]%
\begin{center}
\includegraphics[scale=.65]{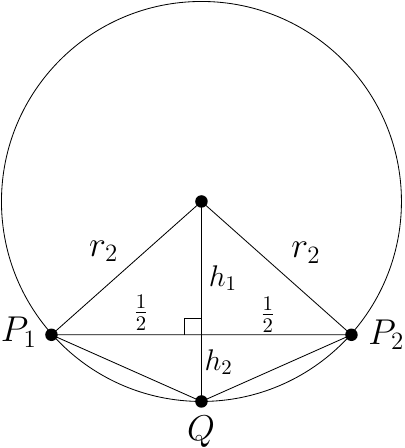}
\caption{}%
\label{circle}%
\end{center}
\end{figure}

\noindent We have the relationships $h_1 + h_2 = r_2$, $h_1 = \sqrt{r_2^2 - \frac14}$, and $|P_1 - Q|^2 = |P_2 - Q|^2 = h_2^2 + \frac14$.  We claim that $|P_1 - Q|^2 < 1$ which amounts to showing that $h_2 < \frac{\sqrt3}{2}$.  To see this, combine the above equalities to write $h_2 = r_2 - \sqrt{r_2^2 - \frac14}$.  Letting $f(x) = x - \sqrt{x^2 - \frac14}$, we have $f'(x) = 1 - \frac{x}{\sqrt{x^2 - \frac14}} < 0$ which implies $f(x)$ is decreasing.  Since $\frac{\sqrt2}{2} < r_2$, and $f(\frac{\sqrt2}{2}) = \frac{\sqrt2 - 1}{2} < \frac{\sqrt3}{2}$, we have established that there is a point on $S$ at a distance less than 1 from each of $P_1$ and $P_2$.  To see that there is a point on $S$ at a distance greater than 1 from each of $P_1$ and $P_2$, just take an endpoint of the diameter of $S$ that is orthogonal to the plane containing this great circle.  The distance from this point to each of $P_1$ and $P_2$ is $r_2\sqrt{2}$ which is greater than 1.  This completes proof of our claim.  We note also that a similar argument shows that for any two points on $S$ that are a distance less than 1 apart, there is a point on $S$ at distance 1 from each of them as well.

Now, to embed a cycle $C_m$ on $S$, we perform the following procedure.  If $m$ is odd, place $w_1, w_3, w_5, \ldots, w_m$ on a great circle of $S$ where $w_1$ and $w_m$ are a Euclidean distance 1 apart, and $w_3, \ldots, w_{m-2}$ lie on the arc of the great circle between $w_1$ and $w_m$.  For each pair of consecutive vertices in $\{w_1, w_3, \ldots, w_m\}$, there is a point on $S$ at distance one from each of them.  We may select these points to be the $w_2, w_4, \ldots, w_{m-1}$ and we have completed the embedding.  If $m$ is even, we may embed the cycle $C_{m-1}$ in the fashion as just described, then delete the edge $w_1w_2$, and place new edges $w_1P$ and $Pw_2$ where vertex $P$ is a point on $S$ at distance 1 from each of $w_1$ and $w_2$.  This completes the proof that $G$ is representable in $\mathbb{R}^{n+2}$.

Now suppose to the contrary that $G$ is embeddable in $\mathbb{R}^{n+1}$.  In such an embedding, we again have that $a_1, \ldots, a_n$ must be represented as the vertices of a regular $(n-1)$-dimensional simplex of edge-length 1 which we may freely assume is centered at the origin.  It follows that each of $w_1, \ldots, w_m$ must be represented as $\mathbb{R}^{n+1}$ points of the form $(0, \ldots, 0, x_i, y_i)$ where $x_i^2 + y_i^2 = \frac{n + 1}{2n}$.  In other words, the cycle $C_m$ must have a unit-distance representation on a circle of radius $\sqrt{\frac{n + 1}{2n}}$.  This contradicts Lemma \ref{nocycle}.\qed
\end{proof}

We are now ready for the main result of this section.

\begin{theorem} Let $n \geq 2$, and $c \in \mathbb{Z}^+$.  Then there exists a dimension-critical graph $H$ satisfying $\dim(H) = n+2$ and $|E(H)| > c$.
\end{theorem}

\begin{proof} Again, consider the graph $G = K_n + C_m$ where $m > c$, and label the vertices of $G$ as in the proof of Theorem \ref{knpluscm}.  For any edge of the form $w_iw_j$, there is an automorphism of $G$ mapping that edge to $e = w_1w_m$.  We aim to show then that $e$ is critical to the dimension of $G$ -- in other words, that $\dim(G) > \dim(G - e)$.  In light of Theorem \ref{knpluscm}, this amounts to showing that $G - e$ is representable in $\mathbb{R}^{n+1}$.

Just as in the proof of Theorem \ref{knpluscm}, we represent $a_1, \ldots , a_n$ as vertices of a regular $(n-1)$-dimensional simplex which is centered at the origin and has radius $r_1 = \sqrt{\frac{n-1}{2n}}$.  Each of the vertices $w_1, \ldots , w_m$ will be points of the form $(0, \ldots, 0, x_i, y_i)$ where $x_i^2 + y_i^2 = \frac{n + 1}{2n}$ for $i \in \{1, \ldots , m\}$.  To see that this does indeed give a valid representation of $G - e$ in $\mathbb{R}^{n+1}$, we need only show that a path of arbitrary length has a unit-distance embedding on a circle, call it $S$, of radius $r_2 = \sqrt{\frac{n + 1}{2n}}$.  Since $r_2 > \frac12$, for any point $p$ on $S$, there are two points on $S$ at distance 1 from $p$.  Since Lemma \ref{nocycle} guarantees that no cycle is embeddable on $S$, we have established that $G - e$ is embeddable in $\mathbb{R}^{n+1}$.  

To then create a dimension-critical graph $H$ with $\dim(H) = n + 2$, start with $G$ and iteratively delete any edges that are not critical to the dimension of the graph.  As observed above, all edges of the form $w_iw_j$ are critical, so no matter how many edges of the form $a_ia_j$ or $a_iw_j$ are deleted, we have that $\dim(H) = n + 2$ and $|E(H)| > c$.\qed 
\end{proof}

\section{Further Work}

In this section we scrape together a few observations and questions that have arisen during our investigations into the topic of dimension-critical graphs.  To the best of our knowledge, each of these are open.  We begin with a question in computational complexity.  A full digression into the terminology, history, and methodology of this subject would take us far afield, so we will make do with assuming some familiarity of our readers, and point those uninitiated to the introductory texts \cite{arora} and \cite{du}.

\begin{question} \label{question1} For an arbitrary graph $G$, what is the complexity of determining whether $G$ is dimension-critical?
\end{question}

In \cite{schaefer}, Schaefer proves that for a general graph $G$, it is NP-complete to determine whether or not $G$ has a unit-distance representation in $\mathbb{R}^2$.  An immediate extension is the fact that it is NP-hard to precisely determine $\dim(G)$.  However, one can also use Schaefer's result to prove that for a given $e \in E(G)$, it is NP-hard to decide if $\dim(G) > \dim(G - e)$.  We do this below.

First, observe that a graph $G$ has a unit-distance representation in $\mathbb{R}$ if and only if $G$ is acyclic and contains no vertices of degree greater than 2 -- in other words, if and only if every component of $G$ is a path.  There are linear-time algorithms for deciding if $G$ has either of these two properties.  Secondly, we note the impossibility of the existence of a graph $H$ with $\dim(H) = 1$ where the creation of a graph $H'$ by placing an edge between two non-adjacent vertices of $H$ results in $\dim(H') > 2$.  This is easy to see considering that $H'$ would have at most one non-path component with that component being a tree, a cycle, or a cycle with one or two paths attached to single vertices of the cycle.  In either case, that component, and by extension the entirety of $H'$, is embeddable in $\mathbb{R}^2$.  

Now suppose to the contrary that there does exist a polynomial-time algorithm to decide whether $e \in E(G)$ is critical to the dimension of $G$.  Label the edges of $G$ as $e_1, \ldots, e_m$ and, starting with $i = 1$, implement this algorithm to decide if $e_i$ is critical.  If it is not, delete $e_i$ from $G$, and implement the algorithm again to decide if $e_{i+1}$ is critical to the dimension of $G \setminus \{e_1, \ldots , e_i\}$.  Eventually we must reach some edge, call it $e_j$, that is critical to the dimension of graph $G' = G \setminus \{e_1, \ldots , e_{j-1}\}$.  We may now run polynomial-time algorithms to decide whether $G'$ is representable in $\mathbb{R}$.  If $\dim(G') = 1$, we have that $\dim(G) = 2$, and if $\dim(G') \neq 1$, we have that $\dim(G) \neq 2$.  The existence of this polynomial-time algorithm to determine whether or not $G$ has a representation in $\mathbb{R}^2$ contradicts Schaefer's result. 

From the above observations, the existence of a polynomial-time algorithm to determine if $G$ is dimension-critical seems very unlikely.  However, we (somewhat abashedly) remark that we see no way to completely resolve Question \ref{question1}.

\begin{question} \label{question2} For an arbitrary graph $G$ and $e \in E(G)$, is it true that $\dim(G) - \dim(G - e) \leq 1$?
\end{question}

Of course, one can produce myriad examples of $G$ and $e \in E(G)$ where the deletion of $e$ either does not change the dimension of the graph or reduces the dimension of the graph by 1.  However, we were unable to find a single instance where $\dim(G) - \dim(G - e) \geq 2$.  Our guess is that such graphs do not exist, and we would be very interested to see a proof.  Incidentally, if one instead considers the deletion of a vertex of $G$, there is a little more that can be said. 

\begin{question} \label{question3} Does there exist an integer $c$ such that for all graphs $G$ and $v \in V(G)$, we are guaranteed to have $\dim(G) - \dim(G \setminus \{v\}) \leq c$?  If so, can we let $c = 2$?
\end{question}

Again, it is easy to construct examples of $G$ and $v \in V(G)$ where $\dim(G) - \dim(G \setminus \{v\})$ is equal to 0 or 1.  However, if we let $G$ be the graph $K_2 + C_6$, we have by Theorem \ref{knpluscm} that $\dim(G) = 4$.  Designating by $v$ one of the vertices of $G$ of degree 7, we have that $G \setminus \{v\}$ is isomorphic to $W_6$.  The wheel $W_6$ is embeddable in $\mathbb{R}^2$ with the usual representation of a regular hexagon of edge-length 1 along with a vertex placed at its center, so here, $\dim(G) - \dim(G \setminus \{v\}) = 2$.  We were unable to construct an example where $\dim(G) - \dim(G \setminus \{v\}) \geq 3$.

\begin{question} \label{question4} For which $n$ does there exist an arbitrarily large bipartite graph $G$ which is dimension-critical with $\dim(G) = n$?
\end{question}

The question above is the easiest non-trivial case of a very deep question that we will present at the end of this section, yet even it appears to be rather thorny.  In \cite{eht}, Erd\H{o}s, Harary, and Tutte demonstrate that for any graph $G$, $\dim(G) \leq 2\chi(G)$ where $\chi(G)$ denotes the vertex-chromatic number of $G$.  It follows that any bipartite graph $G$ has $\dim(G) \in \{1,2,3,4\}$  Note that $\dim(G) = 1$ if and only if every component of $G$ is isomorphic to a path or an isolated vertex, so Question \ref{question4} is trivially answered in the negative when $n = 1$.  Equally trivial is the case $n = 2$ where Question 4 is answered in the affirmative.  Just take an arbitrarily large even cycle as the desired $G$.  For $n = 3$, we will show that it has an affirmative answer in the theorem below.

\begin{theorem} \label{mobiusladdertheorem} There exist arbitrarily large bipartite graphs $G$ which are dimension-critical with $\dim(G) = 3$.
\end{theorem}

\begin{proof} For an integer $n \geq 2$, define the \textit{M\"{o}bius Ladder} $M_{2n}$ to be the graph of order $2n$ constructed by beginning with two copies of the path $P_n$, say with the standard vertex sets $\{a_1, \ldots, a_n\}$ and $\{b_1, \ldots, b_n\}$, respectively, and then placing the additional edges $a_ib_i$ for $i \in \{1, \ldots, n-1\}$ along with $a_1b_n$ and $a_nb_1$.  As a reference, $M_{10}$ is drawn in Figure \ref{m10} below.

\begin{figure}%
\begin{center}
\includegraphics[scale=.95]{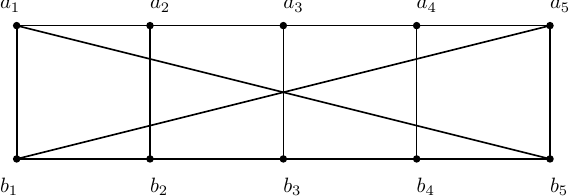}
\caption{}
\label{m10}
\end{center}
\end{figure}

Note that $M_{2n}$ is bipartite when when $n$ is odd.  We will now show that $\dim(M_{2n}) > 2$.  Indeed, suppose to the contrary that $M_{2n}$ has been drawn as a unit-distance graph in $\mathbb{R}^2$.  In such a representation, for $i = 1,2, \ldots, n-1$, the vertices $a_i, a_{i+1}, b_i, b_{i+1}$ form the vertices of a rhombus.  Since opposite sides of a rhombus are parallel, the vector with initial point $a_{i+1}$ and terminal point $b_{i+1}$ must be a translate of the vector with initial point $a_i$ and terminal point $b_i$.  Without loss of generality, we may assume that in this supposed unit-distance drawing of $M_{2n}$ in $\mathbb{R}^2$, we have $a_1$ placed at the origin and $b_1$ placed at $(1,0)$.  Now consider circles $C_{a_1}$ and $C_{b_1}$ drawn in Figure \ref{twocircles}, each of radius 1 and centered at $(0,0)$ and $(1,0)$, respectively.  Since $a_1b_n \in E(G)$, we must have $b_n$ placed on $C_{a_1}$, and similarly, $a_n$ placed on $C_{b_1}$.  However, by the rationale we described above, the line segment connecting $a_n$ and $b_n$ must be horizontal with $a_n$ to the left and $b_n$ to the right.  This is a contradiction as it would force $a_n$ to be placed in the exact same position as $b_1$ (as well as $b_n$ being placed in the same position as $a_1$).       

\begin{figure}%
\begin{center}
\includegraphics[scale=.55]{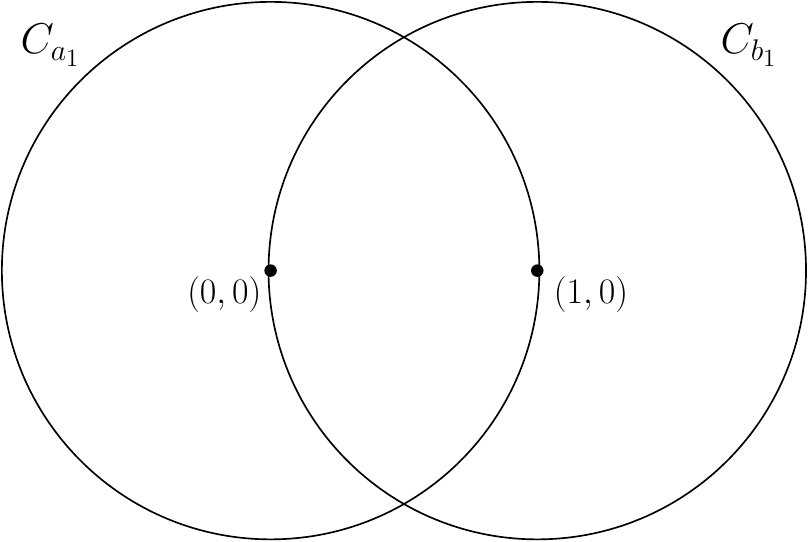}
\caption{}
\label{twocircles}
\end{center}
\end{figure}

When $n = 3$, the graph $M_{2n}$ is isomorphic to $K_{3,3}$ and it has already been seen that $\dim(K_{3,3}) = 4$.  For all higher $n$, it is the case that $\dim(M_{2n}) = 3$ and furthermore, $M_{2n}$ is dimension-critical.  However, we do not need this fact to establish proof of the theorem.  One need only observe that, should one start with $M_{2n}$ and then delete vertices (if necessary) until a dimension-critical subgraph $H$ of $M_{2n}$ has been created with $\dim(H) = 3$, then for each $i \in \{1, \ldots, n\}$, the vertices $a_i$ and $b_i$ would not both be deleted.  The theorem immediately follows.\qed
\end{proof}

We have been unable to resolve Question \ref{question4} when $n = 4$.  In fact, other than $K_{3,3}$, we have not been able to supply any concrete examples of dimension-critical bipartite graphs $G$ with $\dim(G) = 4$.  As it turns out, though, such graphs do exist, which can be seen by observing two major results in extremal combinatorics.  In \cite{brown}, Brown constructs a family of bipartite graphs of order $n$ which do not have $K_{3,3}$ as a subgraph, and whose number of edges is asymptotically on the order of $n^{\frac53}$.  It is independently shown by Kaplan, Matou\v{sek}, Safernov\'{a}, and Sharir in \cite{kaplan} and by Zahl in \cite{zahl} that an upper bound for the number of edges in a graph $G$ of order $n$ and satisfying $\dim(G) = 3$ is asymptotically on the order of $n^{\frac32}$.  Thus for sufficiently large $n$, a graph $G$ of order $n$ produced via Brown's construction will automatically satisfy $\dim(G) = 4$.  Unfortunately (at least, from our point of view), Brown's construction is entirely algebraic, and it seems quite difficult to determine what a dimension-critical subgraph of this $G$ would actually be.

The general formulation of Question \ref{question4} is given below.

\begin{question} \label{question5} For which $n, k$ does there exist an arbitrarily large dimension-critical graph $G$ with $\chi(G) = k$ and $\dim(G) = n$?
\end{question}

A full resolution of this question is far beyond our present reach.  For example, a torrent of work has been produced in the past few years on coloring unit-distance graphs in $\mathbb{R}^2$, much of it stemming from de Grey's stunning construction \cite{degrey} of a 5-chromatic graph unit-distance graph in the plane.  Yet still, it is unknown as to whether there even exists $G$ satisfying $\dim(G) = 2$ and $\chi(G) \in \{6,7\}$, let alone an arbitrarily large dimension-critical $G$ with those properties.  However, if a successful approach could resolve Question 4, perhaps it could be applied to the more modest Question 6.

\begin{question} \label{question6} For which $k$ does there exist an arbitrarily large dimension-critical graph $G$ with $\chi(G) = k$ and $\dim(G) = 2k$?
\end{question}


\end{document}